\documentclass[11pt]{amsart}%fleqn
\usepackage{fullpage}

\usepackage{amsmath}
\usepackage{amssymb}
\usepackage{amsfonts}
\usepackage{graphicx}
\usepackage{amsthm}
\usepackage{enumerate}
\usepackage{lscape}
\usepackage{dsfont}
\usepackage{color}
\usepackage{mathtools}
\usepackage[dvipsnames]{xcolor}
\usepackage{hyperref}
\usepackage{appendix}

\newcounter{example}
\newenvironment{example}[1][]{\refstepcounter{example}\par\medskip
   \noindent \textbf{Example~\theexample. (#1)} \rmfamily}{\medskip}    

\usepackage[utf8]{inputenc}

\newcommand{\C}{\mathds C}
\newcommand{\R}{\mathds R}

\newcommand{\Vol}{\operatorname {Vol}}

\newcommand{\CP}{\mathds C {\rm P}}

\usepackage{apptools}

\newtheorem{theorem}{Theorem}

\newtheorem{lemma}{Lemma}
\newtheorem{proposition}{Proposition}
\newtheorem{corollary}{Corollary}

\newtheorem{remark}{Remark}

\numberwithin{equation}{section}

\def\subsubsection{\@startsection{subsubsection}{3}%
\z@{.5\linespacing\@plus.7\linespacing}{-.5em}%
{\normalfont\bfseries}}
\makeatother

\title{Geometric quantization and TYCZ coefficients of hyperK\"ahler $4$-manifolds}

\author{Simone Cristofori}

\subjclass[2020]{32Q20, 58C05, 58C25}%53C07, 53C42}
\keywords{HyperK\"ahler manifolds, geometric quantization, $\varepsilon$-function, TYCZ expansion}

\begin{document}

\maketitle

\begin{abstract}
In this paper we study conditions for the existence of a geometric quantization of a complete hyperK\"ahler four manifold with an effective tri--Hamiltonian circle action and we show that the vanishing of the third coefficient $a_3$ associated to these metrics is equivalent to flatness.
\end{abstract}

\section{Introduction}
Let $M$ be a $n$-dimensional complex manifold endowed with a K\"ahler metric $g$ and assume that there exists a \emph{polarization} of $M$, namely a holomorphic line bundle $L\rightarrow M$ over $M$ whose first Chern class satisfies 
\[
c_1(L)=\bigg[\frac{\omega}{2\pi}\bigg],
\]
where $\omega$ is the  fundamental form associated to $g$. This condition is equivalent to require the K\"ahler form $\omega/2\pi$ to be integral. Given a \emph{polarized} K\"ahler metric, one can define an hermitian product $h$ on $L$ with curvature form ${\rm Ric}(h):=i\Theta_h$ equal to $\omega$, where $\Theta_h$ is the curvature of the Chern connection of the hermitian holomorphic line bundle $(L,h)$.
For each positive integer $m\in\mathds Z^+$, the hermitian metric $h$ on $L$ induces an hermitian metric $h_m$ on the $m$-th tensor power $L^m:=L^{\otimes m}$ that satisfies ${\rm Ric}(h_m)=m\omega$. Recall that ${\rm Ric}(h_m)$ is the $2$-form on $M$ whose local expression in a trivializing open set $U\subset M$ is given by:
$$
{\rm Ric}(h_m)=- i \partial\bar \partial \log h_m(\sigma(x),\sigma(x)),
$$
for a trivializing holomorphic section $\sigma:U\rightarrow L\setminus \{0\}$. The pair $(L^m,h_m)$ is called a \emph{geometric quantization} of the K\"ahler manifold $(M,m\omega)$.

Consider the separable complex Hilbert space $\mathcal H_m$ of global holomorphic sections $s$ of $L^m$ that are $L^2$-limited in norm with respect to the product:
$$
\langle s,s\rangle_{m}:=\int_Mh_m(s(x),s(x))\frac{\omega^n}{n!}.
$$
Let $\{s_j\}_j$, $j=0,\dots, d_m$ (where $\dim \mathcal H_m=d_m+1$) be an orthonormal basis of $\mathcal H_m$. We observe that when $M$ is compact, $\mathcal H_m$ reduces to the space $H^0(L^m)$ of global holomorphic sections of $L^m$ which has $0<\dim H^0(L^m)<\infty$ for $m$ sufficiently large. However in the noncompact case, $\mathcal H_m$ could reduce to be $\{0\}$ or its dimension could be infinite.
When $\mathcal H_m\neq \{0\}$, we can 
%pick an orthonormal basis $\{s_j\}$ of $\mathcal H_m$ and 
define the following smooth function:
$$
\varepsilon_{m g}(x):=\sum_{j=0}^{d_m}h_m(s_j(x),s_j(x)).
$$ 
This function is globally defined on $M$ and it depends only on the K\"ahler metric $mg$ and not on the orthonormal basis chosen or on the hermitian metric (see e.g. \cite{cgr1}), and in literature it also appears under the name of {\em distortion function}.

When $M$ is compact, D. Catlin \cite{catlin} and independently S. Zelditch \cite{zelditch} proved that the function $\varepsilon_{m g}$ admits an asymptotic expansion in the $C^\infty$ category, known as Tian-Yau-Catlin-Zelditch (from now on TYCZ) expansion:
\begin{equation}\label{eq:sviluppo asintotico}
\epsilon_{m g}(x)\sim\sum_{j=0}^\infty a_j(x) \bigg(\frac{m}{2\pi}\bigg)^{n-j},
\end{equation}
that is for every pair of nonnegative integers $l$, $r$, there exists a constant $C(l,r)>0$ such that:
\begin{equation}\label{asympt}
\left|\left|\epsilon_{m g}(x)-\sum_{j=0}^l a_j(x) \bigg(\frac{m}{2\pi}\bigg)^{n-j} \right|\right|_{C^r}\le C(l,r)\bigg(\frac{m}{2\pi}\bigg)^{n-l-1}.
\end{equation}
Here $a_0(x)\equiv 1$ and the $a_j(x)$, $j=1,2,\dots$ are smooth functions on $M$.
In \cite{lu} Z. Lu, by
means of Tian’s peak section method, computed the first three coefficients of this expansion and he proved that each of the coefficients $a_j(x)$ is a polynomial of the curvature and its covariant derivatives of the metric $g$ which can be determined by finitely many algebraic operations. 

When $M$ is noncompact, we say that a TYCZ--expansion \eqref{eq:sviluppo asintotico} exists if \eqref{asympt} holds for any compact subset of $K\subset M$, as introduced in \cite{arezzoloi} by C. Arezzo and A. Loi. We observe that in this case, if the metric $g$ can be described by a global K\"ahler potential $\varphi:M\to\R$, i.e. $\omega=\frac{i}{2}\partial\overline\partial\varphi$ or the holomorphic line bundle $L$ that gives the quantization of $M$ is trivial, 
the construction make sense also for nonintegers values of $m$, that will be denoted by $\alpha\in \mathds R^+$. In this case the Hilbert space $\mathcal{H}_{\alpha}$ equals the weighted Hilbert
space $\mathcal{H}_{\alpha\varphi}$ consisting of square integrable holomorphic functions on $M$, with weight $e^{-\alpha\varphi}$,
namely
\[
\mathcal{H}_{\alpha\varphi}=\bigg\{f(z)\in Hol(M)\;|\;\int_X |f|^2 e^{-\alpha\varphi}\frac{\omega^{n}}{n!}<\infty\bigg\}.
\]
If $\mathcal{H}_{\alpha\varphi}\neq\{0\}$, we can take an orthonormal basis $\{f_j^\alpha\}$ so that the epsilon function reads as
\[
\epsilon_{\alpha g}(x)=e^{-\alpha\varphi(x)} K_{\alpha\varphi}(x,x),
\]
where $K_{\alpha\varphi}(x,y)$ is the reproducing kernel of $\mathcal{H}_{\alpha\varphi}$.
In constrast with the compact setting, the existence of an expansion of the epsilon function is not guaranteed. In \cite{engliscoeff} M. Engli\v{s} proved the existence of a TYCZ--expansion in the case of strongly pseudoconvex domains of $\mathds C^n$ with real analytic boundary, and  proved that the coefficients are the same as those computed by Lu for compact manifolds. A more general result can be obtained by the work of X. Ma and G. Marinescu \cite[Thm. 6.1.1]{mamarinescu}, as described in the sequel.

Studying metrics with the TYCZ coefficients being prescribed is a very natural generalization of the problem of finding K\"ahler metrics with constant scalar curvature on a K\"ahler manifold. 
 In this regard Z. Lu and G. Tian \cite{LT} proved that the PDEs $a_j = f$ (with $j \ge 2$ and $f$ a smooth function on $M$) are elliptic. They also showed that if the logterm of the Bergman and Szeg\"o kernel
of the unit disk bundle over $M$ vanishes then $a_k = 0$, for $k > n$ (where $n$ is the complex dimension of $M$). The study of these PDEs is meaningful independently of the existence of a TYCZ expansion. Thus, given any K\"ahler manifold $(M,g)$ it is natural to call the $a_j$’s the \emph{coefficients associated to metric $g$}. 
 
The vanishing of the coefficients $a_k$ for $k\ge n+1$ turns out to be related to some important problems in the theory of pseudoconvex manifolds (cf. \cite{LT,ALZ,LUZ}).

In the noncompact setting, one can find in \cite{LZmama} a characterization of the flat metric among locally hermitian symmetric spaces as the only one with vanishing $a_1$ and $a_2$, while in \cite{FT} Z. Feng and Z. Tu solve a conjecture formulated in \cite{Z} by showing that the complex hyperbolic space is the only Cartan-Hartogs domain where the coefficient $a_2$ is constant.

The first three coefficients of the asymptotic expansion of the $\epsilon$-function, computed by Z. Lu \cite{lu} for compact manifolds and by M. Engli\v s \cite{engliscoeff} for noncompact ones are:
\begin{equation}
\begin{cases} \label{coeffespan} a_1=-\frac12\sigma\\ a_2=-\frac13\Delta_g\sigma+\frac1{24}\left(||R||^2-4||{\rm Ric}||^2+3\sigma^2\right)\\
    \begin{split}
        a_3=&-\frac18\Delta_g\Delta_g\sigma+\frac1{24}{\rm div}{\rm div}(R,{\rm Ric})-\frac16{\rm div}({\rm div}(\sigma {\rm Ric}))+\frac1{48}\Delta_g(||R||^2-4||{\rm Ric}||^2+8\sigma^2)+\\
        &-\frac1{48}\sigma(\sigma^2+||R||^2-4||{\rm Ric}||^2){-}\frac1{24}(\sigma_3({\rm Ric})-{\rm Ric}(R,R){+}R({\rm Ric},{\rm Ric}))
    \end{split}
    \end{cases}
\end{equation}
where $R$ denotes the Riemannian tensor of $g$, given in local coordinates by:
\[
R_{i\overline jk\overline l}=\frac{\partial^2 g_{i\overline j}}{\partial z_k\partial\overline z_l} -\sum_{p,q=1}^n g^{p\overline q}\frac{\partial g_{i\overline q}}{\partial z_k}\frac{\partial g_{p\overline j}}{\partial\overline z_l},
\]
${\rm Ric}$ denotes the Ricci tensor, that (in contrast with Lu's notation that has the opposite sign) reads:
$$
{\rm Ric}_{i\bar j}=g^{l\bar m}R_{i\bar jl\bar m}
$$
and:
$$
\sigma=g^{i\bar j}{\rm Ric}_{i\bar j}
$$
is the scalar curvature. All the norms are taken with respect to $g$, and we further are using the following notations (here the signs have been changed accordingly to our notation):
\begin{equation}\label{lu}
    \begin{split}
        & |D'\sigma|^2=g^{j\overline{i}}\frac{\partial\sigma}{\partial z_i}\frac{\partial\sigma}{\partial \overline{z_j}},\qquad  |D'{\rm Ric}|^2=g^{\alpha\overline{i}}g^{j\overline{\beta}}g^{\gamma\overline{k}}Ric_{i\overline{j},k}  \overline{{\rm Ric}_{\alpha\overline{\beta},\gamma}},\\
        &|D'R|^2=g^{\alpha\overline{i}}g^{j\overline{\beta}}g^{\gamma\overline{k}}g^{l\overline{\delta}}g^{\epsilon\overline{p}}R_{i\overline{j}k\overline{l},p}\overline{R_{\alpha\overline{\beta}\gamma\overline{\delta},\epsilon}},\\
        &
        {\rm div}{\rm div}(\sigma {\rm Ric})=2|D'\sigma|^2+g^{\beta\overline{i}}g^{j\overline{\alpha}}{\rm Ric}_{i\overline{j}}\frac{\partial^2\sigma}{\partial z_{\alpha}\partial\overline{z_\beta}}+\sigma\Delta_g\sigma,\\
        &
           \sigma_3({\rm Ric})=g^{\delta\overline{i}}g^{j\overline{\alpha}}g^{\beta\overline{\gamma}}{\rm Ric}_{i\overline j}{\rm Ric}_{\alpha\overline\beta}{\rm Ric}_{\gamma\overline \delta},\\
        &
        R({\rm Ric},{\rm Ric})=g^{\alpha \overline{i}}g^{j\overline\beta}g^{\gamma \overline k}g^{l\overline\delta}R_{i\overline jk\overline l}{\rm Ric}_{\beta\overline{\alpha}}{\rm Ric}_{\delta\overline{\gamma}},\\
        &
        {\rm Ric}(R,R)=g^{\alpha \overline i}g^{j\overline\beta}g^{\gamma \overline k}g^{\delta \overline p}g^{q\overline\epsilon}{\rm Ric}_{i\overline j}R_{\beta\overline\gamma p \overline q}R_{k\overline\alpha\epsilon\overline\delta},\\
       &{\rm div}({\rm div}(R,{\rm Ric}))=-g^{\beta\overline{i}}g^{j\overline{\alpha}}Ric_{i\overline{j}}\frac{\partial^2\sigma}{\partial z_{\alpha}\partial\overline{z_\beta}}-2|D'{\rm Ric}|^2{-}g^{\alpha\overline{i}}g^{j\overline{\beta}}g^{\gamma\overline{k}}g^{l\overline{\delta}}{\rm Ric}_{i\overline{j},k\overline{l}}R_{\beta\overline{\alpha}\delta\overline{\gamma}}+\\
        &\qquad\qquad\qquad\quad\  -R({\rm Ric},{\rm Ric}){+}\sigma_3({\rm Ric}),
    \end{split}
\end{equation}
where $",p"$ represents the covariant derivative with respect to $\frac{\partial}{\partial z_p}$ and $\Delta$ represents the Laplace operator
\begin{equation}\label{laplace}
\Delta_g=\sum_{i,j=1}^n g^{i\overline j}\frac{\partial^2}{\partial z_i\partial\overline{z}_j}.
\end{equation}
In particular, for a K\"ahler--Einstein manifold $(M^n,g)$ with Einstein constant $\lambda$, we have
\[
{\rm Ric}=\lambda g,\qquad 
\sigma=n\lambda
\]
and, as shown in \cite{CZ}, the above coefficients simplify as
\begin{equation*}
\begin{cases}
  a_1=-\frac{n\lambda}{2}\\
  a_2=\frac1{24}\left(||R||^2+n\lambda^2(3n-4)\right)\\
  a_3=\frac1{48}\left(\Delta_g||R||^2-\lambda(n-2)(\lambda^2n(n-2)+||R||^2)\right).
\end{cases}
\end{equation*}

In this paper we study the family of Ricci-flat K\"ahler metrics constructed via the Gibbons--Hawking ansatz on the total space of a principal $S^1$-bundle $X$ over an open set $U\subset\R^3$ (usually $\R^3\backslash\{p_1,\dots,p_N\}$ with $N\le\infty$, called \emph{centers} of the construction) firstly introduced in \cite{GH}. Observe that these metrics generalize the Taub-NUT metrics on $\C^2$ (whose detailed description can be found in \cite{LB}) and the Eguchi--Hanson metric on the blow up of $\C^2$ at the origin (introduced in \cite{EH}). More specifically, this construction produces hyperK\"ahler manifolds of dimension $4$ which are invariant under a circle action.

The first result of this paper is the following, in which we study the existence of a geometric quantization for these metrics:

\begin{theorem}\label{th: quant intro}
    %The K\"ahler manifold $(X,\omega_k)$ for $k=1,2,3$ 
    Let $(X,\omega_1,\omega_2,\omega_3)$ be a connected complete hyperK\"ahler four-manifold with an effective tri-Hamiltonian circle action. For each $l\in\{1,2,3\}$, the K\"ahler manifold $(X,\omega_l)$ admits a geometric quantization if and only if 
    \begin{equation}\label{integer points}
         p_j^l-p_i^l\in\mathbb{Z},\quad\forall i,j.
    \end{equation}
    In particular, the three K\"ahler forms $\omega_1,\omega_2,\omega_3$ admit simultaneously a geometric quantization (with 3 different line bundles) if and only if the centres $p_i$ belong to an affine lattice, i.e. if and only if there exists $c\in\R^3$ such that
    \[
    p_i\in c+\mathbb{Z}^3,\quad \text{for every $i$}.
    \]
\end{theorem}

In \cite{TwoConjectures}, Loi, Salis, and Zuddas, conjecture that a Ricci–-flat metric on a $n$-dimensional complex manifold such that $a_{n+1} = 0$ is forced to be flat. In the case of complex surfaces, this conjecture is equivalent to prove that a Ricci–-flat surface with harmonic $|R|^2$ is flat. 

Our second result characterizes the flat metric among the hyperK\"ahler four manifolds of the Gibbons--Hawking family in terms of the vanishing of the $a_3$ coefficient, giving evidence of this conjecture. 
For a Ricci--flat K\"ahler metric $g$ we have
\[
a_3=\frac{1}{48}\Delta_g||R||^2.
\]
The quantity $a_3$ may be regarded, in the first instance, as the corresponding universal polynomial in the curvature tensor and its covariant derivatives. Whenever the TYCZ-expansion of the epsilon function exists, this polynomial coincides with the third coefficient of the expansion.
More precisely we prove that

\begin{theorem}\label{th: a3 intro}
    Let $(X,g)$ be a complete hyperK\"ahler four manifold with an effective tri-Hamiltonian circle action. Then the third coefficient $a_3$ associated to the metric $g$ is identically zero if and only if $g$ is the flat metric.    
\end{theorem}

The paper is organized as follows. In Section \ref{sec: GH ansatz} we recall what we need about the Gibbons--Hawking construction. In Section \ref{sec: geo quant} we study the geometric quantization for these metrics proving Theorem \ref{th: quant intro}, while the last Section is devoted to the computations of the $a_3$ coefficient proving Theorem \ref{th: a3 intro}.

\section{Gibbons--Hawking ansatz}\label{sec: GH ansatz}

The Gibbons--Hawking ansatz is a technique to produce gravitational instantons introduced by G. Gibbons and S. Hawking in \cite{GH}. A \emph{gravitational instanton} is a complete $4$-dimensional Riemannian manifold which is a solution to the Einstein equations. In particular this ansatz produces $4$-dimensional hyperK\"ahler metrics with a tri-Hamiltonian $S^1$-action. In this section we describe this construction giving primary examples.  

Let $U\subset \R^3$ be an open subset of $\R^3$ with the Euclidean metric and set coordinates $u_1,u_2,u_3$ on $U$. Let $\pi:X\to U$ be a principal $S^1$-bundle with $S^1$-action $S^1\times X\to X$ given by $(e^{it},x)\mapsto e^{it}\cdot x$, where $t$ denotes the coordinate along the fibers, normalized to have period $2\pi$, with $\frac{\partial}{\partial t}$ the corresponding vector field of the $S^1$-action. Let $\theta$ be a connection $1$-form on $X$, i.e. a $i\R$-valued 1-form invariant under the circle action such that $\theta(\frac{\partial}{\partial t})=i$. Then there exists a $2$-form $\alpha$ on $U$ such that the curvature of the connection is  $d\theta=\pi^*\alpha$ and such that the first Chern class of the principal $S^1$-bundle $\pi:X\to U$ is $c_1(X)=[\frac{i}{2\pi}\alpha]$ and it is integral.
Let $V:U\to\R_+$ be a positive function on $U$ that satisfies the abelian monopole equation 
\begin{equation}\label{eq:dV}
\star dV=-i\alpha,    
\end{equation}
where $\star$ is the Hodge star operator on $\R^3$ (note that $V$ is harmonic since $d\alpha=0$, implying $\star d\star dV=0$). Let 
\[
\omega_1=du_1\wedge\frac{\theta}{i}+Vdu_2\wedge du_3
\]
\[
\omega_2=du_2\wedge\frac{\theta}{ i}+Vdu_3\wedge du_1
\]
\[
\omega_3=du_3\wedge\frac{\theta}{ i}+Vdu_1\wedge du_2.
\]
We observe that 
\begin{itemize}
    \item $\omega_1^2=\omega_2^2=\omega_3^2$ is nowhere zero,
    \item $\omega_i\wedge\omega_j=0$ for $i\neq j$,
    \item $d\omega_i=0$ for all $i$ (since \eqref{eq:dV} and $\star
    dV=\partial_x V du_2\wedge du_3+\partial_y V du_3\wedge du_1+\partial_z V du_1\wedge du_2$).
\end{itemize}
Thus $\omega_1,\omega_2,\omega_3$ define an hyperK\"ahler metric on the total space $X$.

Let $\theta_0=\frac{\theta}{i}$ be the real connection $1$-form, then
\[
\theta_0\bigg(\frac{\partial}{\partial t}\bigg)=1,\quad\int_{S^1}\theta_0=2\pi,
\]
and notice that
\[
-\omega_1-i\omega_2=(\theta_0-iVdu_3)\wedge(du_1+idu_2).
\]

By taking this to be the (holomorphic) $2$-form $\Omega$ on $X$, one obtains an integrable almost complex structure on $X$, where $du_1 + idu_2$ and $\theta_0-iVdu_3$ span the holomorphic cotangent space inside the complexified cotangent space. It follows that the (integrable)
almost complex structure $J$ on the cotangent space is given by
\[
J^*(du_1)=-du_2,\quad J^*(du_3)=-V^{-1}\theta_0.
\]
Thus, on vectors the almost complex structure is given by ($(J^*\alpha)X=\alpha(JX)$)
\[
J(\partial_{u_1})=\partial_{u_2},\quad J(\partial_{u_3})=V\theta_0^{\#}
\]
where $\theta_0^{\#}$ is the dual vector of $\theta_0$.

If we consider the K\"ahler form $\omega=\omega_3$ as an alternating tensor, and use the relation that if $g$ is the Riemannian metric, then $g(X,Y)=\omega(X,JY)$, we obtain an expression for the metric
\[
g_X=V\sum_{i=1}^3 du_i^2+V^{-1}\theta_0^2.
\]

Conversely, taken a positive harmonic function $V:U\to\R_+$ such that $-\frac{\star dV}{2\pi}$ represents $c_1(X)$, then we can always find a connection $1$-form $\theta$ such that $d\theta=i\star dV$ (such $\theta$ up to pullbacks of closed $1$-forms from $U$), obtaining hyperK\"ahler metrics as above. 

\begin{remark}\label{rem: curvature}
   By Cartan's method the curvature of this metric reads as
   \[
   |R|^2=12V^{-6}|\nabla V|^4+V^{-4}\Delta(|\nabla V|^2)-6V^{-5}(\nabla V)\cdot(\nabla(|\nabla V|^2))
   \]
   and by (Bochner's identity and) the fact that $V$ is harmonic, it reduces to
   \[
   |R|^2=\frac{1}{2}V^{-1}\Delta\Delta V^{-1},
   \]  
   where $\Delta$ is the standard Laplacian defined on $\R^3$.
\end{remark}

The basic examples of metrics arising from the Gibbons-Hawking ansatz are the Taub-NUT metrics on $\C^2$ and the Eguchi--Hanson metric. In this context they arise as follows:

\begin{example}[Taub-NUT metrics on $\C^2$]\label{ex:Taub-NUT}
Consider the map 
\[
p:X=\C^2\backslash\{(0,0)\}\to \R^3\backslash\{(0,0,0)\}
\]  
defined by
\[
(z_1,z_2)\mapsto(2\text{Re}(z_1 z_2),2\text{Im}(z_1z_2),|z_1|^2-|z_2|^2)
\]
obtained by composing the Hopf fibration with complex conjugation on $z_2$. 
This map exhibits $X$ as an $S^1$-bundle over $\R^3\backslash\{(0,0,0)\}$, with Chern class $\pm1$.The action of $S^1$ on $X$ is given by $e^{it}\cdot(z_1,z_2)=(e^{it}z_1,e^{-it}z_2)$.

We choose a positive harmonic function on $\R^3\backslash\{(0,0,0)\}$ such that the Chern number is $\pm 1$, i.e. 
\[
-\frac{1}{2\pi}\int_{S^2} \star dV=\frac{1}{2\pi}\int_{S^2} i\alpha=\pm 1.
\]
A such $V$ is 
\[
V_e=e+\frac{1}{2|u|}=e+\frac{1}{2\sqrt{u_1^2+u_2^2+u_3^2}},
\]
with $e\ge0$ a nonnegative constant.  

We take as connection form 
\[
\theta=i\frac{\text{Im}(\overline{z_1}dz_1-\overline{z_2}dz_2)}{(|z_1|^2+|z_2|^2)}
\]
and we have $d\theta=i\star dV$. We obtain in this way hyperK\"ahler metrics on $X=\C^2\backslash\{(0,0)\}$, which for all $e\ge0$ extend to $\C^2$. In fact, such metrics are ALF (asymptotically locally flat), approaching a flat metric when $|u|\to \infty$, whilst being periodic in $t$. When $e=0$, this is the flat metric on $\C^2$.   
\end{example}

\begin{example}[Eguchi--Hanson metric] Let $U=\R^3\backslash\{e_1,-e_1\}$ with $e_1=(1,0,0)\in\R^3$ and $V:U\to\R$ the harmonic function given by
\[
V(u)=\frac{1}{2}\sum_{i=1}^2\frac{1}{|u-p_i|},\quad\text{with $p_i\in\{e_1,-e_1\}$}.
\]
The Gibbons--Hawking ansatz gives the Eguchi--Hanson (or Calabi metric) on the cotangent bundle $T^*\CP^1$ of $\CP^1$. This metric on $X$ extends smoothly over the points $\pi^{-1}(p_1)$ and $\pi^{-1}(p_2)$ and it is ALE (asymptotically locally euclidean), asymptotic to the flat metric on $\R^4\backslash \mathbb{Z}_2$.
\end{example}

Choosing various harmonic functions $V$ in the Gibbons--Hawking construction produces a number of examples of complete hyperK\"ahler $4$-manifolds with tri-holomorphic circle action. 
The complete hyperK\"ahler $4$-manifolds with tri--Hamiltonian symmetry have been classified by Bielawski for the finite topological type in \cite{B} and extended by Swann to the infinite topological type in \cite{Swann}:

\begin{theorem}[Bielawski-Swann,\cite{B},\cite{Swann}]\label{th: BS}
If $M$ is a connected complete hyperK\"ahler four-manifold with an effective tri-Hamiltonian circle action, then $M$ is isometric to a Gibbons-Hawking metric, an Anderson-Kronheimer-LeBrun metric or a Taub-NUT deformation of one of these, namely
\begin{itemize}
    \item[(i)] For $V(u)=\frac{1}{2}\sum_{i=1}^k\frac{1}{|u-p_i|}$ with $p_i\in\R^3$ distinct points, we get the Gibbons--Hawking gravitational instantons;
    \item[(ii)] For $V(u)=\frac{1}{2}\sum_{i=1}^\infty\frac{1}{|u-p_i|}$ with $p_i\in\R^3$ distinct points, such that $V(p)<\infty$ for some $p\in\R^3$, we get the metrics of Anderson, Kronheimer and LeBrun (see \cite{AKL});
    \item[(iii)] Adding a positive constant $c>0$ to either of the previous two potentials we get the Taub-NUT deformations of the metrics.
\end{itemize}
\end{theorem}

\begin{remark}
    Considering the multi-centers metrics arising from the harmonic function
    \[
    V(u)=c+\frac{1}{2}\sum_{i=1}^k\frac{1}{|u-p_i|},
    \]
    we have (see \cite{D})
    \begin{itemize}
        \item for $c=0$ and $k=2$, the Eguchi-Hanson metric and for $c=0$ and $k>0$ the general $A_{k-1}$ gravitational ALE instantons;
        \item for $c>0$ and $k=1$, the Taub-NUT metric and for 
         $c>0$ and $k>1$ the general $A_{k-1}$ gravitational ALF instantons.
    \end{itemize}

    These metrics have finite topological type for $k<+\infty$, that is the second Betti number is $b_2(X)=k-1$, while they have infinite topological type for $k=+\infty$, that is $
    H_2(X,\mathbb{Z})=\bigoplus_{k=1}^\infty \mathbb{Z}$.
\end{remark}

\begin{remark}\label{rem: asymptotic}
    Let $k<+\infty$ be a positive integer and consider the metrics obtained by an harmonic function of the type
    \[
    V(u)=c+\frac{1}{2}\sum_{i=1}^k\frac{1}{|u-p_i|}
    \]
    that gives the multi-Taub-NUT metrics for $c>0$ and the multi-Eguchi-Hanson metrics for $c=0$.
    Let us study the asymptotic growth of the norm of the curvature tensor. 
    
    Let $c>0$.
    Setting $r:=|u|=\sqrt{u_1^2+u_2^2+u_3^2}$, we observe that for $r\to\infty$, each multipole term grows as
    \[
    \frac{1}{|u-p_i|}=\frac{1}{r}+\frac{p_i\cdot u}{r^3}+O(r^{-3})=\frac{1}{r}+\frac{p_i\cdot \frac{u}{|u|}}{r^2}+O(r^{-3}),
    \] 
    so that the function $V$ has asymptotic behavior
    \[
    V(u)= c+\frac{k}{2r}+\frac{1}{2r^2} \sum_{i=1}^k p_i\cdot\frac{u}{|u|}+O(r^{-3}),
    \]
    where the first part is radial in $r$ while the second depends on the direction $u/|u|$. In particular for each $s\ge0$, 
    \[
    \partial^s\big(V-c-\frac{k}{2r}\big)=O(r^{-2-s}).
    \]
    If we write
    \[
    V(u)=c+\frac{k}{2r}+E(u)
    \]
    with $E(u)=O(r^{-2})$ and $\partial^s E(u)=O(r^{-2-s})$, we have
    \[
    V^{-1}(u)=c^{-1}\bigg(1+\frac{1}{c}\bigg(\frac{k}{2r}+E(u)\bigg)\bigg)^{-1}=\frac{1}{c}-\frac{k}{2c^2r}+O(r^{-2}).
    \]
    Since for $r\to\infty$, $\Delta(1/c)=\Delta(1/r)=0$, we have
    \[
    \Delta\Delta V^{-1}= \Delta\Delta\bigg(V^{-1}-\frac{1}{c}+\frac{k}{2c^2r}\bigg).
    \]
    From 
    \[
    \partial^s\bigg(V^{-1}-\frac{1}{c}+\frac{k}{2c^2r}\bigg)=O(r^{-2-s}),
    \]
    we obtain
    \[
    \Delta V^{-1}=O(r^{-4})
    \]
    so that
    \[
    \Delta\Delta V^{-1}=O(r^{-6}).
    \]
    Finally, being $V^{-1}=c^{-1}+O(r^{-1})$, we conclude
    \[
    |R|^2=\frac{1}{2}V^{-1}\Delta\Delta V^{-1}=O(r^{-6}).
    \]
    
    Let now $c=0$ and $\rho\sim\sqrt{|u|}=2\sqrt{r}$ be the geodesic radius on the manifold $X$. In this case we obtain the multi-Eguchi-Hanson metric that is an ALE $A_k$ gravitational instanton. Thus the metric is required to satisfy (see \cite{K}) 
    \[
    g=g_{euc}+O(\rho^{-4}).
    \]
    From this it follows that $g^{-1}=O(1)$. Moreover, 
    \[
    \Gamma_{ij}^k=\frac{1}{2}g^{kj}(\partial_i g_{lj}-\partial_j g_{il}+\partial_l g_{ij})=O(\rho^{-5})
    \]
    and
    \[
    R_{ijkl}=g_{lm}R^m_{ijk}=g_{lm}(\partial_i\Gamma_{jk}^m-\partial_j\Gamma_{ik}^m+\Gamma_{jk}^p\Gamma_{ip}^m-\Gamma_{ik}^p\Gamma_{jp}^m)=O(\rho^{-6}),
    \]
    giving
    \[
    |R|^2=g^{ia}g^{jb}g^{kc}g^{ld}R_{ijkl}R_{abcd}=O(\rho^{-12}),
    \]
    i.e. $|R|=O(\rho^{-6})$.
\end{remark}

\section{Geometric quantization}\label{sec: geo quant}

A geometric quantization $(L,h)$ of a $n$-dimensional K\"ahler manifold $(M,\omega)$ consists of an hermitian holomorphic line bundle $L$ over $M$ such that the first Chern class of $L$ is represented by $\frac{\omega}{2\pi}$ and
its curvature ${\rm Ric}(h):=-{i}\partial\overline\partial\log h$ satisfies ${\rm Ric} = \omega$. Notice that such an $(L,h)$ exists if and only if $[\frac{\omega}{2\pi}]\in H^2(X,\mathbb{Z})$.

 In our case, being the metric hyperK\"ahler, there are three different K\"ahler forms $\omega_1,\omega_2,\omega_3$ and we want to study when a geometric quantization exists, namely study when $[\frac{\omega_j}{2\pi}]\in H^2(X,\mathbb{Z})$ for $j=1,2,3$. We remark that for a K\"ahler form $\omega$ to be integral means that
 \[
 \frac{1}{2\pi}\int_\gamma \omega\in \mathbb{Z},
 \]
 where $\gamma$ are the $2$-cycles that generate the second homology group $H_2(X,\mathbb{Z})$. The Gibbons-Hawking metrics are constructed on a principal $S^1$-bundle $X$ over an open set of $\R^3$, thus the space $X$ is locally $S^1\times\R^3$. 

 Let us consider the metrics obtained by the Gibbons-Hawking ansatz with $V=c+\frac{1}{2}\sum_{j=1}^N\frac{1}{|x-p_j|}$, with $c\ge0$ and $N\le+\infty$ as in Theorem \ref{th: BS}: in this case the open set $U$ is $\R^3\backslash\{p_1,\dots,p_N\}$. 

 As described in \cite{BW}, the multi-center Gibbons-Hawking metrics contain $\frac{1}{2}N(N-1)$ topologically non-trivial two-cycles, $\Delta_{ij}$ (respectively a countable basis of $H_2(X,\mathbb{Z})$ in the infinite topological case), that run between the Gibbons--Hawking centers. These two-cycles can be
defined by taking any curve, $\gamma_{ij}$, between  $p_i$ and $p_j$ and considering the $S^1$-fiber along the curve. This fiber collapses to zero at the Gibbons--Hawking centers, and so the curve and the fiber produce a $2$-sphere $S^2_{ij}$ (up to $\mathbb{Z}_{|q_j|}$ orbifolds). These spheres intersect one another at the common points $p_j$. There are $(N-1)$  linearly independent homology two-spheres, and the set $\Delta_{i(i+1)}$ represents a basis. 

\begin{proof}[Proof of Theorem \ref{th: quant intro}]

Let us consider an oriented smooth curve 
\[
\gamma_{ij}:[0,1]\to\R^3
\]
with 
\[
\gamma_{ij}(0)=p_i,\quad \gamma_{ij}(1)=p_j
\]
 which does not pass through any other center. At every point of the curve, the fiber is a circle $S^1$ while at the border the fiber collapses. We have
\[
S^2_{ij}=\pi^{-1}(\gamma_{ij})\cong S^2.
\]
We parametrize the curve $\gamma_{ij}$ by
\[
\gamma_{ij}(t)=\big(u_1(t),u_2(t),u_3(t)\big),\quad t\in[0,1].
\]
 In particular
\[
du_1=\dot{u_1}(t)dt,\quad du_2=\dot{u_2}(t)dt,\quad du_3=\dot{u_3}(t)dt 
\]
and in the form $\omega_1$, the term $Vdu_2\wedge du_3$ vanishes, being $du_l$ proportional to $dt$. Thus
\[
\omega_1|_{S^2_{ij}}=du_1\wedge\frac{\theta}{ i}
\]
and we have
\[
\int_{S^2_{ij}}\omega_1=\int_{S^2_{ij}} du_1\wedge\frac{\theta}{ i}=\int_{S^1} \frac{\theta}{ i} \int_{\gamma_{ij}}du_1=2\pi(p_j^1-p_i^1).
\]
Similarly for the K\"ahler forms $\omega_2$ and $\omega_3$. Thus, from the classification result of complete $4$-hyperK\"ahler manifolds with a $S^1$-action, we have proved the Theorem.
\end{proof}

 It should be noted that for a general noncompact manifold which admits a geometric quantization, there is not a general theorem which assures the existence of the TYCZ-expansion of the $\epsilon$-function. In this direction sufficient conditions are given by X. Ma and G. Marinescu \cite[Theorem 6.1.1]{mamarinescu}, which translated in our setting reads as (see also \cite[Theorem 5]{CZ1} and \cite[Theorem 3]{CZ})

\begin{theorem}
   Let $(X,g,\omega)$ be a complete K\"ahler manifold. %and let $(L,h)$ be an hermitian holomorphic line bundle on $X$. 
   Then the epsilon function $\epsilon_{m g}$ admits an asymptotic expansion in $m$ with coefficients given by \eqref{coeffespan} provided there exist a constant $c>0$ such that
   \[
    iR^{\det}>-c\omega %,\quad|\partial\omega|_g<c,  iR^L>l\omega ,\quad
   \]
where $R^{\det}$ denotes the curvature of the connection on $\det(T^{1,0}X)$ induced by $g$. 
% and $R^L$ the curvature of the connection on  $L$ induced by the hermitian metric $h$.
\end{theorem}

 Once the Ma–Marinescu Bergman kernel expansion is available under the positivity assumptions above, the leading term of the Bergman kernel is positive. Hence, for all sufficiently large $m$, the corresponding space of square-integrable holomorphic sections cannot be trivial.

As shown in \cite{CZ}, these conditions are satisfied for a K\"ahler--Einstein metric.
In particular, since the metrics constructed by the Gibbons--Hawking ansatz are Ricci-flat, they are K\"ahler--Einstein and combining this Theorem and Theorem \ref{th: quant intro}, we obtain in our situation

\begin{theorem}
    Let $g$ be a complete hyperK\"ahler metric on a principal $S^1$-bundle over an open set $U$ of $\R^3$ satisfying the integrality conditions on the points \eqref{integer points}. Then
\[
\mathcal{H}_m\neq\{0\}
\]
for all sufficiently large $m$, and the corresponding epsilon function $\epsilon_{mg}$ admits a TYCZ expansion in $m$, uniformly on compact subsets.
\end{theorem}

\begin{remark}
     In particular, the Taub--NUT and Eguchi--Hanson metrics satisfy the above assumptions, as shown in \cite{LZZ} and \cite{aghedu}, respectively. In these works, the authors investigate the balanced condition for these metrics and obtain explicit expressions for the $\epsilon$-function by constructing complete orthonormal systems for $\mathcal{H}_m$.
\end{remark}

\section{The $a_3$ coefficient of the Gibbons--Hawking metrics}\label{sec: a3 computations}

We start by computing the third coefficient of the asymptotic expansion of the $\epsilon$-function for the basic examples of metrics arising from the Gibbons--Hawking ansatz, namely the Taub--NUT metric and the Eguchi--Hanson metric (the proofs of the following Propositions contain also the expression of the $a_2$ coefficient which can be also found in the Appendix of \cite{LZZ} and \cite{CZ1} respectively where these metrics arise from different constructions). In particular, we recover the classification result of the flat metric as the only one in the Taub-NUT family with vanishing $a_3$ (given in \cite{LZZ}). 

\begin{remark}%[Conventions for $\Delta_g$ and curvature norms]
We briefly clarify that for a K\"ahler metric, we denote by
\[
\Delta_g=\sum_{i,j}g^{i\overline j}\partial_i\partial_{\overline j}
\]
the complex Laplacian. In real coordinates, the corresponding Laplace--Beltrami operator on the underlying Riemannian manifold is
\[
\Delta_g^{LB}=2\Delta_g.
\]
Likewise, if $| R|^2=R_{abcd}R^{abcd}$ denotes the norm of the Riemann curvature tensor computed with the real Riemannian metric, then for a K\"ahler metric
\[
|R|^2=4||R||^2,
\]
where $||R||^2$ is the complex norm used in the curvature coefficients above. Consequently, in the Ricci-flat case,
\[
a_2=\frac{1}{24}||R||^2
=\frac{1}{96}|R|^2,
\]
and
\[
a_3=\frac{1}{48}\Delta_g||R||^2
=\frac{1}{192}\Delta_g|R|^2
=\frac{1}{384}\Delta_g^{LB}|R|^2.
\]
%These conversion factors should be taken into account when comparing our computations, which are written in the Gibbons--Hawking parametrization and using real curvature norms, with formulas in the literature written in different parametrizations or conventions. In particular, a direct numerical comparison of the displayed formulas is meaningful only after the corresponding parametrizations and normalizations of the metric have been identified.
\end{remark}

First we note the following:

\begin{lemma}
Let \(U\subset\mathbb{R}^3\) be an open subset and let
\(f:U\to\mathbb{R}\) be a smooth function. Then, on
\(\pi^{-1}(U)\), we have
$$
\Delta^{LB}_g(\pi^*f)
=
\pi^*\bigl(V^{-1}\Delta f\bigr),
$$
where \(\Delta\) denotes the standard Euclidean Laplacian on
\(\mathbb{R}^3\).
\end{lemma}

\begin{proof}
Since \(\pi^*f\) is \(t\)-invariant, we have $\partial_t(\pi^*f)=0$.
Therefore, the mixed terms \(g^{i0}\) in the Laplace--Beltrami
operator do not contribute. Since 
$$
\sqrt{\det g}=V
\qquad\text{and}\qquad
g^{ij}=V^{-1}\delta^{ij},
$$
we obtain
$$
\begin{aligned}
\Delta^{LB}_g(\pi^*f)
&=
\frac{1}{\sqrt{\det g}}
\partial_\alpha
\left(
\sqrt{\det g}\,
g^{\alpha\beta}
\partial_\beta(\pi^*f)
\right)\\
&=
\frac{1}{V}
\partial_i
\left(
V\,V^{-1}\delta^{ij}\partial_j f
\right)\\
&=
\frac{1}{V}\sum_{i=1}^3\partial_i^2f\\
&=
\pi^*\bigl(V^{-1}\Delta f\bigr).
\end{aligned}
$$
\end{proof}

In what follows, we identify a function $f$ on the base with its
$t$-invariant lift $\pi^*f$ to $X$. Thus, we will simply write
\[
\Delta_g^{LB} f=V^{-1}\Delta f
\]
without explicitly indicating the pullback.

\begin{proposition}
 Let $e\ge0$ be a real number and $g_e$ be the Taub--NUT metric on $\C^2$ given by the harmonic function $V_e$ (of Example \ref{ex:Taub-NUT}). The third coefficient of the asymptotic expansion of the epsilon function vanishes iff $g_e=g_0$, i.e. iff $g$ is the flat metric on $\C^2$.
\end{proposition}
\begin{proof}
Being the metrics constructed by Gibbons-Hawking ansatz Ricci--flat and by Lemma 1 of \cite{CZ}, we just need to compute $\Delta_g^{LB}|R|^2$. First we have
\[
\Delta\Delta V_e^{-1}=\frac{192 e^2 }{|u|(1+2e|u|)^5}.
\]
By the remark \ref{rem: curvature}, it follows that
\[
2|R|^2=\frac{384 e^2}{(1+2e|u|)^6}.
\]
Then by the lemma above
\[
2\Delta_g^{LB}|R|^2=\frac{18432e^3(5e|u|-1)}{(1+2e|u|)^9}
\]
which vanishes identically iff $e=0$, that is iff $g$ is the flat metric on $\C^2$.
\end{proof}

In the following Proposition we compute the $a_3$ coefficient of the Eguchi--Hanson metric: 

\begin{proposition}
    The $a_3$ coefficient of the Eguchi--Hanson metric is not identically zero.
\end{proposition}
\begin{proof}
    As in the proof above, we compute the norm of the Riemann tensor for $g_V$ with
    \[
    V(x)=\frac{1}{2|x-e_1|}+\frac{1}{2|x+e_1|}=\frac{1}{2\sqrt{(x_1-1)^2+x_2^2+x_3^2}}+\frac{1}{2\sqrt{(x_1+1)^2+x_2^2+x_3^2}}.
    \]
    The bilaplacian of the inverse of $V$ is
    \[
    \Delta\Delta(V^{-1})=\frac{768}{|x-e_1||x+e_1|(|x-e_1|+|x+e_1|)^5},
    \]
    so that
    \[
    2|R|^2=\frac{1536}{(|x-e_1|+|x+e_1|)^6}
    \]
    and 
    \[
    2\Delta_g^{LB}|R|^2=\frac{36864(-9+5|x|^2+5|x-e_1||x+e_1|)}{(|x-e_1|+|x+e_1|)^9},
    \]
    giving $ 2\Delta^{LB}_g|R|^2|_{0}=-288$, that is $a_3(0)=\frac{1}{384}\Delta_g^{LB}|R|^2=-3/8$.
\end{proof}

\begin{proposition}
    Let $f:\R^3\to\R$ be an harmonic function. Then we have:
    \begin{itemize}
        \item[(i)] $\Delta\big(\frac{1}{f}\big)=2\frac{|\nabla f|^2}{f^3}$;
        \item[(ii)] $\frac{1}{2}\Delta\Delta\big(\frac{1}{f}\big)=\frac{2|D^2f|^2}{f^3}+\frac{12|\nabla f|^4}{f^5}-\frac{12(D^2f)\nabla f\cdot \nabla f}{f^4}$,
    \end{itemize}
    where $\nabla$ denotes the gradient operator and $D^2$ the hessian of the function.
\end{proposition}
\begin{proof}
    For $(i)$, the gradient of the inverse of $f$ is
    \[
    \nabla\bigg(\frac{1}{f}\bigg)=-\frac{1}{f^2}\nabla(f)
    \]
    and the Leibniz rule give
    \[
    \Delta\bigg(\frac{1}{f}\bigg)=-\nabla\cdot \bigg(\frac{1}{f^2}\nabla(f)\bigg)=-\nabla\bigg(\frac{1}{f^2}\bigg)\cdot \nabla f-\frac{1}{f^2}\Delta f=\frac{2}{f^3}\nabla(f)\cdot \nabla(f)-\frac{1}{f^2}\Delta f,
    \]
    obtaining $(i)$ being $f$ harmonic.
    
    For $(ii)$, using the Leibniz rule for the Laplacian
    \[
    \Delta(fg)=g\Delta f+f\Delta g+2\nabla f\cdot\nabla g
    \]
    with $f=|\nabla f|^2$ and $g=f^{-3}$ and computing first that
    \[
    \nabla(f^n)=nf^{n-1}\nabla f,
    \]
    \[
    \Delta(f^n)=n(n-1)f^{n-2}|\nabla f|^2+nf^{n-1}\Delta f,
    \]
    \[
    \nabla(|\nabla(f)|^2)=2(D^2f)\nabla f,
    \]
    \[
    \Delta(|\nabla f|^2)=2|D^2 f|^2+2\nabla f\cdot \nabla(\Delta f),
    \]
    we obtain $(ii)$ by simplifying with $f$ harmonic. 
\end{proof}

Applying this Proposition to the harmonic function  $f=V$  of the Gibbons--Hawking ansatz, it follows that we can express the norm of the Riemannian tensor as
\begin{equation}\label{eq: curv}
    |R|^2=\frac{1}{2}\frac{1}{V}\Delta\Delta\bigg(\frac{1}{V}\bigg)=\frac{2|D^2V|^2}{V^4}+\frac{12|\nabla V|^4}{V^6}-\frac{12(D^2V)\nabla V\cdot \nabla V}{V^5}.
\end{equation}

The argument underlying Theorem \ref{th: a3 intro} is in fact more general and does not rely essentially on the Gibbons–Hawking structure. The latter is only used to ensure the decay condition
\[
\inf_X ||R||^2=0.
\]
We therefore isolate the following general statement.

\begin{proposition}\label{prop: liouville}
Let $(X,g)$ be a complete Ricci-flat K\"ahler manifold of complex dimension $2$, and assume that
\[
\inf_X ||R||^2=0.
\]
If $a_3\equiv 0$, then $g$ is flat.
\end{proposition}
\begin{proof}
    Since $g$ is Ricci-flat and has complex dimension two,
\[
a_3(g)=\frac1{48}\Delta_g||R||^2.
\]
Assume that $a_3(g)\equiv0$. Then
\[
\Delta_g||R||^2=0.
\]
Consequently
\[
u:=1+||R||^2
\]
is a positive harmonic function on a complete Riemannian manifold $(X,g)$ such that $\inf_X ||R||^2=0$. 
 By Yau's Liouville theorem (\cite{Yau}), $||R||^2$ must be constant and thus $||R||^2\equiv0$ and the metric is flat.
\end{proof}

The following two lemmas provide the verification of the assumption $\inf_X ||R||^2=0$ in the Bielawski--Swann class of hyperK\"ahler $4$-manifolds. 
In the first one we give a curvature estimate via the distance from the centres while in the second we prove that this distance is unbounded. 

\begin{lemma} 
Let 
\[
d(x):=\inf_j|x-p_j|
\]
denote the distance from $x$ to the set of centres $p_j$. Then the curvature satisfies 
\[
||R||^2=O(d(x)^{-2}).
\]
\end{lemma}
\begin{proof}
For 
\[
V(x)=c+\frac12\sum_j\frac1{|x-p_j|}
\]
we have the following estimates:
\[
|\nabla V|\le\frac{V}{d(x)}
\]
and
\[
|D^2V|\le\frac{\sqrt6\,V}{d(x)^2}.
\]
In fact, letting  $r_j=|x-p_j|$,
\[
|\nabla V|
\le\frac12\sum_jr_j^{-2}
\le\frac1{d(x)}\frac12\sum_jr_j^{-1}
\le\frac{V}{d(x)}.
\]
Moreover
\[
\left|D^2\frac1r\right|^2=\frac6{r^6},
\]
and then
\[
|D^2V|
\le\frac{\sqrt6}{2}\sum_jr_j^{-3}
\le\frac{\sqrt6\,V}{d(x)^2}.
\]
By \eqref{eq: curv}, it follows that
\[
||R||^2
\le
\frac{24+12\sqrt6}{V^2d(x)^4}.
\]
If there exists at least one center,
\[
V(x)\ge\frac1{2d(x)},
\]
and hence
\[
||R||^2(x)
\le
\frac{96+48\sqrt6}{d(x)^2}.
\]
\end{proof}

\begin{lemma}  
Assume that there exists a point $q_0\in\mathbb R^3$ such that
\[
\sum_j\frac{1}{|q_0-p_j|}<\infty.
\]
Then the distance function $d$ from the set of centers is unbounded on $\mathbb R^3$. In particular, there exists a sequence $x_\nu\in\mathbb R^3$ such that
\[
d(x_\nu)\longrightarrow\infty,
\]
and hence
\[
\lVert R\rVert^2(x_\nu)\longrightarrow0.
\]
\end{lemma}

\begin{proof}
By translation of coordinates, we can assume $q_0=0$ and set
\[
A:=\sum_j\frac1{|p_j|}<\infty.
\]
Let
\[
N(R):=\#\{j:|p_j|\le R\}.
\]
Since for $|p_j|\le R$ it holds
\[
\frac1{|p_j|}\ge\frac1R,
\]
we have
\[
\frac{N(R)}R
\le
\sum_{|p_j|\le R}\frac1{|p_j|}
\le A,
\]
namely
\begin{equation}\label{eq: N(R)}
    N(R)\le AR.
\end{equation}
Suppose by contradiction that the distance to the set of centers is uniformly bounded:
\[
d(x)\le D
\qquad\text{for every }x\in\mathbb R^3.
\]
Then for every $x\in B_R(0)$ there exists a center $p_j$ such that $|x-p_j|<D$. Consequently,
\[
|p_j|\le |p_j-x|+|x|\le D+R,
\]
and hence
\[
B_R(0)
\subset
\bigcup_{|p_j|\le R+D}B_D(p_j).
\]
By \eqref{eq: N(R)}, the number of balls on the right-hand side is of order $O(R)$. Therefore, the volume of their union is at most $O(R)$,
while
\[
\Vol B_R(0)=O(R^3).
\]
Letting $R\to\infty$ yields a contradiction.
Thus
\[
\sup_{x\in\mathbb R^3}d(x)=+\infty.
\]
Hence, there exists a sequence $x_\nu$ such that
\[
d(x_\nu)\longrightarrow\infty.
\]
By the above Lemma
\[
||R||^2(x_\nu)\longrightarrow0.
\]
In the case of finitely many centers, the same conclusion follows immediately by choosing a sequence tending to infinity.
\end{proof}

We are now in position to prove Theorem \ref{th: a3 intro}.

\begin{proof}[Proof of Theorem \ref{th: a3 intro}]
By the Bielawski--Swann classification, away from the fixed points of the circle action the metric is of Gibbons--Hawking type with
\[
V(x)=c+\frac12\sum_j\frac1{|x-p_j|},
\qquad c\ge0,
\]
where the set of centers is either finite or satisfies the convergence condition
\[
\sum_j|q_0-p_j|^{-1}<\infty
\]
for some $q_0\in\mathbb R^3$.

If there are no centers, the Gibbons--Hawking potential is constant and the metric is already flat. 

Otherwise, if $a_3\equiv0$, the result follows from Proposition \ref{prop: liouville}. 
The converse is immediate.
\end{proof}

\begin{remark}
    We notice that for the complete hyperK\"ahler four manifold with an effective tri-Hamiltonian circle action of finite topological type, the proof of the above Theorem can follow directly by the asymptotic growth of the norm of the curvature tensor given in Remark \ref{rem: asymptotic} applying the same result of Yau (\cite{Yau}).
\end{remark}

We conclude this section giving an explicit computation in the case of  metrics of infinite topological type given in Theorem \ref{th: BS}, namely the Anderson-Kronheimer-LeBrun metrics and their Taub-NUT deformations,  obtained from the Gibbons--Hawking ansatz with
\[
V(x)=\frac{1}{2}\sum_{i=1}^\infty\frac{1}{|x-p_i|},
\]
where $p_i\in\R^3$ are distinct points with the condition that for some point $p_0$ we have $\frac{1}{2}\sum_{i=j}^\infty\frac{1}{|p_0-p_j|}<\infty$. As in \cite{AKL}, we can take the sequence of points $\{p_j\}_{j=1}^\infty$ in $\R^3$ to be $p_j=j^2e_1=(j^2,0,0)$ with $j\ge1$. 
These metrics are no longer ALE.

\begin{proposition}\label{prop: curvature AKL}
    Let $g$ be the metric of Anderson-Kronheimer-LeBrun. Then 
    \[
    |R|^2_{|_0}\neq0.
    \]
\end{proposition}
\begin{proof}
    Let $f=\frac{1}{2|x-j^2e_1|}=\frac{1}{2\sqrt{(x-j^2)^2+y^2+z^2}}$. We compute that
    \[
    f(0)=\frac{1}{2j^2},\quad \partial_x f_{|_0}=\frac{1}{2j^4},\quad
    \partial^2_x f_{|_0}=\frac{1}{j^6},\quad \partial_y f_{|_0}=\partial_z f_{|_0}=0,\quad \partial^2_y f_{|_0}=\partial^2_z f_{|_0}=-\frac{1}{2j^6},\quad \partial_{xy}f_{|_0}=0.
    \]
    The operator $\nabla$ and $D^2$ are linear. Thus we have
    \[
    \frac{12|\nabla V|^4}{V^6}|_0=\frac{12((1/2)+(1/32)+(1/162)+\dots)^4}{(1/2+1/8+1/18+\dots)^6}=\frac{12(\sum_{j=1}^\infty \frac{1}{2j^4})^4}{(\sum_{j=1}^\infty \frac{1}{2j^2})^6}=\frac{\frac{3}{4}(\zeta(4))^4}{(\frac{\pi^2}{2\cdot6})^6}=\frac{\frac{3}{4}\cdot (\frac{\pi^4}{90})^4}{(\frac{\pi^2}{2\cdot6})^6}=\frac{64\pi^4}{1875};
    \]
     \[
    \frac{2|D^2V|^2}{V^4}|_0=\frac{2((\partial^2_x V)^2+(\partial_y^2 V)^2+(\partial_z^2 V)^2)}{(\sum_{j=1}^\infty \frac{1}{2j^2})^4}=\frac{2((\sum_{j=1}^\infty \frac{1}{j^6})^2+2(\sum_{j=1}^\infty\frac{1}{2j^6})^2)}{(\frac{\pi^2}{2\cdot6})^4}=\frac{2(\zeta(6)^2+\frac{1}{2}\zeta(6)^2)}{(\frac{\pi^2}{2\cdot6})^4}=\frac{3\zeta(6)^2}{(\frac{\zeta(2)}{2})^4}
    \]
    \[
    \frac{12(D^2V)\nabla V\cdot \nabla V}{V^5}|_0=\frac{12(\partial^2_x V(\partial_x V)^2)}{(\sum_{j=1}^\infty \frac{1}{2j^2})^5}=\frac{12(\sum_{j=1}^\infty \frac{1}{j^6})(\sum_{j=1}^\infty\frac{1}{2j^4})^2}{(\frac{\pi^2}{2\cdot6})^5}=\frac{3\zeta(6)\zeta(4)^2}{(\frac{\zeta(2)}{2})^5}=\frac{256\pi^4}{2625};
    \]
    where $\zeta(n)$ is the Riemann zeta function in $n$. Thus,
    \[
    |R|^2_{|_0}=\frac{64\pi^4}{1875}+\frac{256\pi^4}{3675}-\frac{256\pi^4}{2625}=\frac{192 \pi ^4}{30625}
    %\frac{64 \left(149 \pi ^4-12600\right)}{91875}\sim 1.33326.
    \]

\end{proof}

\begin{corollary}\label{cor: curvature AKl def}
    The Taub-NUT deformations of the Anderson-Kronheimer-LeBrun metric arising from 
    \[
    V(x)=c+\frac{1}{2}\sum_{j=1}^\infty \frac{1}{|x-p_j|}
    \]
    has not vanishing curvature in zero.
\end{corollary}
\begin{proof}
    Being $\nabla, D^2$ differential operators, adding a positive constant $c$ to $V$, affects only the value of the denominators in \eqref{eq: curv}. In particular we get
    \begin{equation*} 
    \begin{split}
    |R_c|^2_{|_0}&=\frac{\frac{3}{4}\zeta(4)^4}{(c+\frac{\xi(2)}{2})^6}+\frac{3\zeta(6)^2}{(c+\frac{\zeta(2)}{2})^4}-\frac{3\zeta(6)\zeta(4)^2}{(c+\frac{\zeta(2)}{2})^5}\\
    &=\frac{192 \pi ^{12} \left(40 c+\pi ^2\right)^2}{30625 \left(12 c+\pi ^2\right)^6}.
    \end{split}
     \end{equation*}
     Moreover we observe that $|R_c|^2\to0$ as $c$ grows.
\end{proof}

Now, as an explicit example, we construct in this case a sequence of points satisfying the hypothesis $\inf_X ||R||^2=0$ of Proposition \ref{prop: liouville}: let 
    \[
    V_c(x)=c+\frac12\sum_{j=1}^\infty \frac1{|x-j^2e_1|}
    \]
    be the harmonic function that give rise to the infinite topological type metrics. 
    From Proposition \ref{prop: curvature AKL} and Corollary \ref{cor: curvature AKl def} we have that $|R_c|^2$ is positive in zero for every $c\ge0$. 

    We exhibit a sequence $q_R$ escaping to infinity along which $|R_c|(q_R)\to0$ as $R\to\infty$. This will conclude the proof by the result of Yau (\cite{Yau}) as in Theorem \ref{th: a3 intro}.

    Let $q_R=(0,R,0)\in\R^3$. We have that
    \[
    V_0(q_R)=\frac{1}{2}\sum_{j=1}^\infty \frac{1}{\sqrt{j^4+R^2}},
    \]
    \[
    \nabla V_0(q_R)=\bigg(\frac{1}{2}\sum_{j=1}^\infty \frac{j^2}{(j^4+R^2)^{3/2}},-\frac{1}{2}\sum_{j=1}^\infty \frac{R}{(j^4+R^2)^{3/2}},0\bigg),
    \]
    \[
    D^2V_0(q_R)=\begin{pmatrix}
    \sum_{j=1}^\infty \frac{2j^4-R^2}{2(j^4+R^2)^{5/2}} &
   - \sum_{j=1}^\infty \frac{3j^2R}{2(j^4+R^2)^{5/2}}&0\\
    - \sum_{j=1}^\infty \frac{3j^2R}{2(j^4+R^2)^{5/2}}& 
     -\sum_{j=1}^\infty \frac{j^4-2R^2}{2(j^4+R^2)^{5/2}} &0\\
     0&0&
    -\sum_{j=1}^\infty \frac{1}{2(j^4+R^2)^{3/2}} 
    \end{pmatrix}.
    \]
    Set $h=R^{-\frac{1}{2}}$ and rewrite the sums defining $V_0(q_R)$ and its derivatives as Riemann sums in the variable 
    $t = jh$. Since the corresponding rescaled functions are integrable on $(0,\infty)$, one obtains 
    \[
    V_0(q_R)\asymp R^{-1/2},  
    \]
    \[
    \frac{1}{2}\sum_{j=1}^\infty \frac{j^2}{(j^4+R^2)^{3/2}}\asymp -\frac{1}{2}\sum_{j=1}^\infty \frac{R}{(j^4+R^2)^{3/2}}\asymp R^{-3/2},
    \]
    \[
    \sum_{j=1}^\infty \frac{2j^4-R^2}{2(j^4+R^2)^{5/2}}\asymp
   - \sum_{j=1}^\infty \frac{3j^2R}{2(j^4+R^2)^{5/2}}\asymp
   -\sum_{j=1}^\infty \frac{j^4-2R^2}{2(j^4+R^2)^{5/2}}\asymp
   -\sum_{j=1}^\infty \frac{1}{2(j^4+R^2)^{3/2}}\asymp R^{-5/2},
    \]       
    where we write $a\asymp b$ if $a$ and $b$ are comparable up to positive multiplicative constants, i.e., if there exist $C_1,C_2>0$ such that
    \[
    C_1b\leq a\leq C_2b,
    \]    
    and it follows that
\[
|\nabla V_0(q_R)|=O(R^{-3/2}),\quad
|D^2V_0(q_R)|=O(R^{-5/2}),\quad (D^2V)\nabla V\cdot\nabla V= O(R^{-11/2}).
\]
Thus by \eqref{eq: curv} we obtain that $|R_0|^2(q_R)=O(R^{-3})$, that is $|R_0|^2\to 0$.

Now, if $c>0$ then $V_c(q_R)\to c$, while $\nabla V_c(q_R)\to 0$ and $D^2V_c(q_R)\to 0$, so that $|R_c|^2(q_R)\to0$.


\begin{thebibliography}{99}

\bibitem{AKL} Anderson, M.T., Kronheimer, P.B., LeBrun, C. Complete Ricci-flat Kähler manifolds of infinite topological type. Commun.Math. Phys. 125, 637–642 (1989). https://doi.org/10.1007/BF01228345.

\bibitem{ALZ} C.Arezzo, A.Loi, F.Zuddas, Szeg\"o kernel,regular quantizations and spherical CR-structures, {\em Math.Z.} {\bf275} (2013), 1207-1216.

\bibitem{arezzoloi} C. Arezzo, A. Loi, Moment maps, scalar curvature and quantization of K\"ahler manifolds, {\em Comm. Math. Phys.} {\bf243} (2004) 543–559.

\bibitem{BW} Bena, I. and. Warner, N.P (2008) Black Holes, Black Rings, and Their Microstates. In: Bellucci, S., Ed., Supersymmetric Mechanics-Vol. 3. Lecture Notes in Physics, Vol. 755, Springer, Berlin, 1-92.

\bibitem{B}  R. Bielawski, Complete hyper-Kähler 4n-manifolds with a local tri-hamilton
ian Rn-action, Math. Ann. 314 (1999), no. 3, 505–528

\bibitem{cgr1} M. Cahen, S. Gutt, J. H. Rawnsley,
{\em Quantization of K\"{a}hler manifolds I: Geometric
interpretation of Berezin's quantization}, JGP. 7 (1990), 45-62.

\bibitem{aghedu} F. Cannas Aghedu, On the balanced condition for the Eguchi-Hanson metric, J. Geom. Phys. 137 (2019), 35-39.

\bibitem{catlin} D. Catlin, The Bergman kernel and a theorem of Tian, in Analysis and geometry in several
complex variables (Katata, 1997), Trends Math., Birkhuser Boston, Boston, MA (1999), 1-23

\bibitem{CZ}  Cristofori, Simone; Zedda, Michela. On the third coefficient in the TYCZ–expansion of the epsilon function of Kähler–Einstein manifolds / - In: JOURNAL OF GEOMETRY AND PHYSICS. - ISSN 0393-0440. - (2025).

\bibitem{CZ1} Cristofori, S., Zedda, M., Kähler Geometry of Scalar Flat Metrics on Line Bundles Over Polarized Kähler–Einstein Manifolds /  - In: THE JOURNAL OF GEOMETRIC ANALYSIS. - ISSN 1050-6926. - 34:6(2024). [10.1007/s12220-024-01590-0]

\bibitem{D} Dunajski, M. (2025). Gravitational Instantons, Old and New. Acta Physica Polonica B, 55, Article 12. https://doi.org/10.5506/aphyspolb.55.12-a3

\bibitem{EH} Eguchi, T., Hanson A. J. Self-dual solutions to Euclidean gravity. Ann. Physics 120 no.
1, 82-106 (1979)

\bibitem{engliscoeff} M. Engli\v{s}, The asymptotics of a Laplace integral on a K\"ahler manifold. {\em J. Reine Angew. Math.} {\bf528}, 1--39 (2000).

\bibitem{FT} Z.Feng, Z.Tu, On canonical metrics on Cartan-Hartogs domains, Math. Z. {\bf278} (1-2)(2014), 301-320.

\bibitem{GH}  Gibbons, G. W. and Hawking, S. W. (1978) Gravitational Multi- Instantons. Phys. Lett. B78 430.

\bibitem{Gross-Wilson} Gross, Mark and Pelham M. H. Wilson. “Large Complex Structure Limits of K3 Surfaces.” Journal of Differential Geometry 55 (2000): 475-546. 

\bibitem{K} Kronheimer, P.B. The construction of ALE spaces as hyper-Kähler quotients. J. Differ. Geom. 1989, 29, 665–683, https://doi.org/10.4310/jdg/1214443066.

\bibitem{LB}  C. LeBrun, Complete Ricci-flat K¨ahler metrics on Cn need not be flat, Proceedings
of Symposia in Pure Mathematics, vol. 52 (1991), Part 2, 297-304.

\bibitem{TwoConjectures} A. Loi, F. Salis, F. Zuddas, Two conjectures on Ricci--flat K\"ahler metrics, Math. Zeit., {\bf290} (2018), 599--613.

\bibitem{LUZ} A. Loi, D. Uccheddu, M. Zedda, {On the Szeg\"o kernel of Cartan-Hartogs domains},  {\em Arkiv f\"or Matematik} {\bf54} (2016), n. 2, 473-484.

\bibitem{LZmama} A. Loi, M. Zedda, On the coefficients of TYZ expansion of locally Hermitian symmetric spaces, {\em Manuscripta Math.} {\bf148} (2015), 303–315.


\bibitem{LZZ} Loi, Andrea, Michela Zedda and Fabio Zuddas. “Some remarks on the Kaehler geometry of LeBrun's Ricci flat metrics on $\C^2$.” arXiv: Differential Geometry (2011): n. pag.



\bibitem{lu} Z. Lu, On the lower terms of the asymptotic expansion of Tian-Yau-Zelditch. {\em Am. J. Math.} {\bf122}, 235-273 (2000).

\bibitem{LT} Z. Lu, G. Tian, The Log Term of the Szeg\"o Kernel, Duke Math. J. {\bf 125} (2004), no.2, 351-387.


\bibitem{mamarinescu} X. Ma, G. Marinescu, Holomorphic morse inequalities and Bergman kernels. Progress in Mathematics, Birkhäuser, Basel (2007).



\bibitem{Swann}Swann, Andrew. “Twists versus Modifications.” arXiv: Differential Geometry (2015).

\bibitem{Yau}  T. Yau Harmonic functions on complete Riemannian manifolds, Comm. Pure Appl. Math.
28 (1975), 201-228.

\bibitem{Z} M. Zedda, Canonical metrics on Cartan-Hartogs domains, Int. J. Geom. Methods Mod. Phys. {\bf9} (1) (2012).

\bibitem{zelditch} S. Zelditch, Szeg\"o kernels and a theorem of Tian, {\em Int. Math. Res. Notices} {\bf6}, 317--331 (1998).


\end{thebibliography}
\end{document}